\documentclass[11pt,a4paper]{amsart}

\usepackage[latin1]{inputenc}
\usepackage{amsfonts}
\usepackage{amssymb}
\usepackage{amsthm}
\usepackage{amsmath}
\usepackage{amscd}
\usepackage{t1enc}
\usepackage[mathscr]{eucal}
\usepackage{indentfirst}
\usepackage{mathtools}
\usepackage{cite}
\usepackage{xcolor}
\usepackage{multicol}
\usepackage{graphicx}
\usepackage[margin=1.2in]{geometry}

\usepackage[hyphens]{url}\usepackage{breakurl}\usepackage{hyperref}\hypersetup{pdfstartview={XYZ null null 1.00}}

\newtheorem{theo}{Theorem}
\newtheorem{lemma}{Lemma}

\newtheorem{problem}{Problem}

\newtheorem{THEO}{Theorem}

\newtheorem{CONJEC}{Conjecture}

\begin{document}



\title[On zeros of $X_{1}$ exceptional orthogonal polynomials]{On the behavior of zeros of $X_{1}$-Jacobi and $X_{1}$-Laguerre exceptional orthogonal polynomials}


\author{Yen Chi Lun}
\address{DMAP, IBILCE, UNESP - Universidade Estadual Paulista,
 15055-450 S\~ao Jos\'e do Rio Preto, SP, Brazil}
\email{yen.chilun@yahoo.com.tw}

\begin{abstract}
The $X_1$-Jacobi and the $X_1$-Laguerre exceptional orthogonal polynomials were introduced a decade ago by G\'omez-Ullate, Kamran and Milson in a series of papers.
Their fundamental role in various applications inspired many researcher to study their fundamental properties, including the behavior of their zeros.
In this note we establish some ones, including interlacing and monotonicity of the \textit{regular} zeros with respect to the parameters and the order
of the set of the \textit{exceptional}  zeros as well as the connection between the zeros of the families of $X_{1}$-Jacobi, $X_{1}$-Laguerre and Hermite polynomials.
\end{abstract}

\keywords{$X_1$-Laguerre polynomials; $X_1$-Jacobi polynomials; Exceptional orthogonal polynomials; Zeros}

\maketitle


The $X_{1}$-Jacobi and $X_{1}$-Laguerre exceptional orthogonal polynomials were introduced in 2009 by D.~G\'omez-Ullate, N.~Kamran and R.~Milson \textbf{\cite{GoKaMi09}}.
In a series of papers the same authors established some of the fundamental properties of the polynomials and of their natural generalizations, called the $X_{m}$ exceptional ones,
and immediately attracted the interest of other researchers to the new objects which became widely known as {\em exceptional orthogonal polynomials sequences}, or  succinctly XOPS.

Brief definitions of the sequences $X_{1}$-Jacobi and $X_{1}$-Laguerre exceptional orthogonal polynomials are as follows.
Given $\alpha >0$, the $X_{1}$-Laguerre XOPS, denoted by $\{\widehat{L}_{n}^{(\alpha)}(x)\}_{n=1}^{\infty}$, is the sequence obtained by the Gram-Schmidt orthogonalization process
from the monic polynomials
$
x+\alpha +1, (x+\alpha )^2, (x+\alpha)^3, \ldots
$
with respect to the inner product
$$
\langle P,Q \rangle_{\alpha } := \int_{0}^{\infty }P(x) Q(x)\frac{e^{-x} x^{\alpha}}{(x+\alpha)^{2}}dx.
$$
Similarly, given $\alpha, \ \beta > -1$ with $sign(\alpha)=sign(\beta)$ and $\alpha \neq \beta$, and setting
$$
a=\frac{1}{2}(\beta - \alpha), \ \ b=\frac{\beta + \alpha}{\beta - \alpha} \ \ \mbox{and} \ \ c=b+\frac{1}{a},
$$
the $X_{1}$-Jacobi XOPS, denoted by $\{\widehat{P}_{n}^{(\alpha , \beta )}(x)\}_{n=1}^{\infty}$, is the sequence obtained by the Gram-Schmidt orthogonalization the  polynomials
$
(x-c),(x-b)^2,(x-b)^3, \ldots,
$
with respect to the inner product
$$
\langle P,Q \rangle_{\alpha , \beta} := \int_{-1}^{1}P(x) Q(x) \frac{(1-x)^{\alpha}(1+x)^{\beta}}{(x-b)^{2}}dx.
$$

One of the basic directions of research has been towards understanding the behavior of the zeros of the XOPS and pretty large number of
contributions have already been made, see \textbf{\cite{Bonneux,BonneuxArno,Yen,GoKaMi10,Paco,ABJ,Horvath15,Horvath17}}. Nevertheless,
various questions about the precise location interlacing, sharp inequalities and monotonicity with respect to parameters remain open. In this note
we establish some results concerning the latter issues about the zeros of the $X_{1}$-Jacobi and $X_{1}$-Laguerre polynomials.

Already in their first papers \textbf{\cite{GoKaMi09, GoKaMi10}} G\'omez-Ullate, Kamran and Milson proved that the latter polynomials posses $n-1$ distinct zeros in
the interval of othogonality and an additional zero outside that interval. More precisely, if $\widehat{x}_{n,k}^{(\alpha)}$, $k=1,\ldots,n$, denote the zeros of $\widehat{L}_{n}^{(\alpha)}(x)$ in an increasing order,
then $\widehat{L}_{n}^{(\alpha)}(x)$ has $n-1$ distinct zeros in $(0,\infty)$ and that the remaining zero is in $(-\infty,-\alpha)$. The zeros in $(0,\infty)$ are known as the \textit{regular} zeros and the remaining  one is called the \textit{exceptional} zero.
Similarly, if we denote by $\widehat{x}_{n,k}^{(\alpha,\beta)}$, $k=1,\ldots,n$, the zeros of $\widehat{P}_{n}^{(\alpha , \beta )}(x)$, arranged in an increasing order,
then $\widehat{P}_{n}^{(\alpha , \beta )}(x)$ has $n-1$ distinct zeros in $(-1,1)$ and that the remaining zero is in either $(-\infty,b)$ or $(b,\infty)$, depending on the parameters $\alpha$ and $\beta$. Again, the ones in $(-1,1)$ are called \textit{regular} zeros and the other is the \textit{exceptional} zero.

One of our contributions is a quantitative results concerning a conjecture of A.~B.~J.~Kuijlaars and R.~Milson. It is known that all families of exceptional polynomials $\{p_{n}(x)\}_{n=m}^{\infty}$ have a weight function of the form
$$
W(x)=\frac{W_{0}(x)}{\eta(x)^{2}},
$$
where $W_{0}(x)$ is a classical orthogonal polynomials weight function and $\eta(x)$ is a certain polynomial which does not vanish in the domain of orthogonality and whose degree $m$ is equal to the number of gaps in the degree of XPOS, or equivalently, to the number of \textit{exceptional} zeros. Observe that for the $X_{1}$-Laguerre and $X_{1}$-Jacobi  polynomials $\eta(x)$ is $x+\alpha$ and $x-b$, respectively.
A.~B.~J.~Kuijlaars and R.~Milson \textbf{\cite{ABJ}} formulated the following general conjecture
\begin{CONJEC}\label{conjecture}
The \textit{regular} zeros of XOPS have the same asymptotic behaviour as the zeros of their classical counterpart. The \textit{exceptional} zeros converge to the zeros of the
denominator polynomial $\eta(x)$.
\end{CONJEC}
The statement was partially proved in \textbf{\cite{ABJ}} for the class of Hermite exceptional polynomials introduced in \textbf{\cite{GoGraMi}}. The asymptotic results in Section 6.4 of \textbf{\cite{BonneuxArno}} provide
a firm background in favor of the conjecture for exceptional Laguerre polynomials. The Heine-Mehler type formula for $X_{m}$-Jacobi and $X_{m}$-Laguerre polynomials (see Proposition 3.4 and Proposition 5.6 in \textbf{\cite{Paco}}) leads to the asymptotic behavior of their zeros. Applying these formule to the  $X_{1}$-Jacobi and $X_{1}$-Laguerre polynomials immediately implies the convergence of the \textit{exceptional} zeros of $\widehat{L} _{n}^{(\alpha)}(x)$ and $\widehat{P} _{n}^{(\alpha , \beta )}(x)$ to $-\alpha$ and $b$, respectively.


Our result shows the precise behaviour of the sequences of exceptional zeros $\{ \widehat{x} _{n,n}^{(\alpha , \beta )} \}_{n=1}^{\infty} $ of $X_{1}$-Jacobi and $\{\widehat{x} _{n,1}^{(\alpha)}\}_{n=1}^{\infty}$ of $X_{1}$-Laguerre polynomials. We prove that the latter sequences converge monotonically to their limits $b$ and $-\alpha$ and $\widehat{x} _{n,n}^{(\alpha , \beta )}$ approaches $b$ with a speed at least $O(1/n)$.
\begin{theo}\label{JacobiExceptionalZeroTheo} Let $0<\alpha<\beta$. Then the exceptional zeros $\widehat{x} _{n,n}^{(\alpha , \beta )}$
of the $X_{1}$-Jacobi polynomials belong to $(b,c)$ and the sequence $\{\widehat{x} _{n,n}^{(\alpha , \beta )}\}_{n=1}^{\infty}$ is a strictly decreasing one, that is,
   \begin{equation}\label{zerosinequalities}
  b<\ldots<\widehat{x} _{n,n}^{(\alpha,\beta)}<\ldots<\widehat{x} _{2,2}^{(\alpha , \beta )}<\widehat{x} _{1,1}^{(\alpha , \beta )}=c.
  \end{equation}
Moreover,
\begin{equation}\label{limitantesZeroXJacobi}
  b<\widehat{x} _{n,n}^{(\alpha , \beta )}\leq \frac{2n+\alpha+\beta}{2n-2+\alpha+\beta} b.
  \end{equation}

The exceptional zeros $\widehat{x} _{n,1}^{(\alpha)}$ of the $X_{1}$-Laguerre polynomials belong to $(-\alpha-1,-\alpha)$ and the sequence
$\{\widehat{x} _{n,1}^{(\alpha)}\}_{n=1}^{\infty}$ is a strictly increasing one,
  \begin{equation}\label{LaguerreZerosInequalities}
  -(\alpha+1)=\widehat{x} _{1,1}^{(\alpha)}<\ldots<\widehat{x} _{n-1,1}^{(\alpha )}<\widehat{x} _{n,1}^{(\alpha )}<\ldots<-\alpha.
  \end{equation}
  \end{theo}
Note that in the case $0<\beta<\alpha$, the statement is similar; one needs only to invert the inequalities.

It was shown In \textbf{\cite{Yen}}  that the exceptional zero of the $X_{1}$-Jacobi polynomial $\widehat{P}_{n}^{(\alpha , \beta )}(x)$ is in $(b,c]$ when $\beta > \alpha$ (or in $[c,b)$ when $\beta < \alpha$), that the
regular zeros of $X_{1}$-Jacobi XOPS of consecutive degrees obey the interlacing property as well as that the regular zeros are increasing functions of $\beta$ and decreasing functions of $\alpha$ under certain conditions for these parameters.
Here we establish the analogous  properties of the zeros of $X_{1}$-Laguerre polynomials
\begin{theo}\label{InterlacingMonotZerosLaguerre}
Let the regular zeros $\widehat{x} _{n,k}^{(\alpha)}$, $k=1,\ldots,n$, of $\widehat{L}_{n}^{(\alpha)}(x)$ be arranged in an increasing order. Then the regular zeros $\widehat{x} _{n,k}^{(\alpha)}$, $k=2,\ldots,n$
are increasing functions of $\alpha$. Moreover, the regular zeros of polynomials of consecutive degrees interlace:
$$
0<\widehat{x}_{n+1,2}^{(\alpha)}<\widehat{x}_{n,2}^{(\alpha)}<\widehat{x}_{n+1,3}^{(\alpha)}<\ldots<\widehat{x}_{n+1,n}^{(\alpha)}<\widehat{x}_{n,n}^{(\alpha)}<\widehat{x}_{n+1,n+1}^{(\alpha)}.
$$
\end{theo}

Except for a refinement of the Sturm comparison theorem, due to Dimitrov \textbf{\cite{DKDLate}}, basic tools in the proofs are limiting relations between the zeros of  $X_{1}$-Jacobi, $X_{1}$-Laguerre and Hermite polynomials which might be of independent interest. In the setting for the zeros $x^{(\alpha,\beta)}_{n,k}$, $x^{(\alpha)}_{n,k}$  and $h_{n,k}$
of the corresponding classical Jacobi, Laguerre and Hermite polynomials these relations read as follows.  The one between the zeros of Jacobi and Laguerre polynomials is (see \textbf{\cite{szego}} and \textbf{\cite{Eliel}}  for a recent refinement)
$$
\lim_{\beta\rightarrow\infty}\frac{\beta\big(1-x^{(\alpha,\beta)}_{n,k}\big)}{2}=x^{(\alpha)}_{n,n+1-k},
$$
and the analog for the zeros of Laguerre and Hermite polynomials
$$
 \lim_{\alpha\to\infty}\frac{x_{n,k}^{(\alpha)}-\alpha}{\sqrt{2\alpha}}=h_{n,k}
$$
is due to F.~Calogero \textbf{\cite{Calogero}}. The analogs for the zeros of XOPS is the following one:

\begin{theo}\label{JacobiLaguerreConvergence}
Let $\alpha>0$ and $n,k \in \mathbb{N}$ with $1\leq k \leq n$ be fixed. Then
\begin{equation}\label{JacobiLaguerreFormula}
\lim_{\beta\rightarrow\infty}\frac{\beta\big(1-\widehat{x}^{(\alpha,\beta)}_{n,k}\big)}{2}=\widehat{x}^{(\alpha)}_{n,n+1-k}.
\end{equation}

Similarly, if $n,k \in \mathbb{N}$ with $2\leq k \leq n$ are fixed, then
   \begin{equation}\label{LimX1LaguerreZeroToHermite}
   \lim_{\alpha\to\infty}\frac{\widehat{x}_{n,k}^{(\alpha)}-\alpha}{\sqrt{2\alpha}}=h_{n-1,k-1}
   \end{equation}
\end{theo}


\section{Preliminaries}

 We begin this section recalling a basic fact about relation between the classical  Laguerre orthogonal polynomials and the $X_{1}$-Laguerre XOPS, proved in
 \textbf{\cite{GoKaMi09}},
\begin{equation}\label{relacaoLaguerre}
\widehat{L}^{(\alpha)}_{n}(x)=-(x+\alpha+1)L_{n-1}^{(\alpha)}(x)+L_{n-2}^{(\alpha)}(x),
\end{equation}
and that the $n$-th $X_{1}$-Laguerre XOP is a solution of the differential equation
\begin{equation}\label{eqLaguerre}
x(x+\alpha)y''-(x-\alpha )(x+\alpha +1)y'+ [n x +(n-2)\alpha]y=0,\ \ \ y(x)= \widehat{L}^{(\alpha)}_{n}(x).
\end{equation}


We shall need also the following relations between the classical and the $X_{1}$-Jacobi polynomials, proved in \textbf{\cite{GoKaMi10}}:
\begin{equation}\label{relacaoJacobi}
\widehat{P} _{n}^{(\alpha , \beta )}(x)=-\frac{1}{2}(x-b)P_{n-1}^{(\alpha , \beta )}(x)+\frac{b P _{n-1}^{(\alpha , \beta )}(x)-P _{n-2}^{(\alpha , \beta )}(x)  }{2n-2+\alpha+\beta}
\end{equation}
and
\begin{equation}\label{relacaoJacobi01}
-\frac{1}{4}(x-b)^{2}P_{n}^{(\alpha , \beta )}(x)=f_{n+1}\widehat{P} _{n+2}^{(\alpha , \beta )}(x)-2b \ g_{n}\widehat{P} _{n+1}^{(\alpha , \beta )}(x)+h_{n}\widehat{P} _{n}^{(\alpha , \beta )}(x),
\end{equation}
where
$$
f_{n}=\frac{n(n+\alpha+\beta)}{(2n-1+\alpha+\beta)(2n+\alpha+\beta)},
$$
$$
g_{n}=\frac{(n+\beta)(n+\alpha)}{(2n+2+\alpha+\beta)(2n+\alpha+\beta)},
$$
$$
h_{n}=\frac{(n-1+\beta)(n-1+\alpha)}{(2n+\alpha+\beta)(2n+1+\alpha+\beta)}.
$$

Recall now the classical Sturm Comparison Theorem in Chapter 6 of Szeg\H{o}'s book \cite{szego}:
\begin{THEO}[Sturm's Comparison Theorem]\label{SturmTheoClassic}
Let
$y(x)$ and $Y(x)$ be solutions of the differential equations
\begin{equation*}
\label{y}
y''(x) + f(x) y(x) =0
\end{equation*}
and
\begin{equation*}
\label{Y}
Y''(x) + F(x) Y(x) =0,
\end{equation*}
where $f, F \in C(r,s)$ and $f(x)\leq F(x)$ in $(r,s)$. Let $x_1$ and $x_2$, with $r< x_1 < x_2 < s$ be two
consecutive zeros of $y(x)$. Then the function $Y(x)$ has at least one variation of sign in the interval
$(x_1, x_2)$ provided that $f\not\equiv F$ in $(x_1, x_2)$. The statement also holds:
\begin{itemize}
\item for $x_1=r$ if
\begin{equation}\label{a}
y(r+0)=0 \ \ \mbox{and} \ \ \lim_{x\rightarrow r_{+}} \{y^{\prime}(x) Y(x) - y(x) Y^{\prime}(x) \} =0;
\end{equation}
\item for $x_2=s$ if
\begin{equation}\label{b}
y(s-0)=0\ \ \mbox{and} \ \ \lim_{x\rightarrow s_{-}} \{y^{\prime}(x) Y(x) - y(x) Y^{\prime}(x) \} =0.
\end{equation}
\end{itemize}
\end{THEO}
Note that in the above mentioned Sturm comparison theorem, which we call the classical version, the function $F-f$ does not change of sign in $(r,s)$. Recently D. K. Dimitrov  \textbf{\cite{DKDLate}} proved the following refined version of the classical one where the difference $F-f$ is allowed to change sign only once in the interval $(r,s)$. This modification turned out to work in situations where the classical one does not (see \textbf{\cite{Dimitrov, Driver, Rafaeli}}) and we shall employ it again.
\begin{THEO}[Refined Sturm's Comparison Theorem]\label{SturmTheoNew}
Let $y(x;\tau)$ be the solution of the differential equation
\begin{equation}\label{eq dif s-d}
y^{\prime\prime}(x;\tau)+f(x;\tau) y(x;\tau)=0
\end{equation}
which depends on a parameter $\tau$, where the differentiation is with respect to the variable $x$. Suppose that $f\in\mathrm{C}[(r,s)\times(c,d)]$ and, for every $\tau\in(c,d)$, the solution of \eqref{eq dif s-d} satisfies
\begin{equation}\label{liminextreme}
\lim_{x\rightarrow r_{+}} y(x;\tau)=0 \ \ \mbox{and} \ \ \lim_{x\rightarrow s_{-}} y(x;\tau)=0
\end{equation}
and has $n$ distinct zeros $x_{1}(\tau)<\ldots<x_{n}(\tau)$ in $(a,b)$. Given $\tau_{1},\tau_{2}\in(c,d)$, suppose that $y(x;\tau_{1})$ and $y(x;\tau_{2})$ satisfy the condition \eqref{liminextreme} described above,
$$
\lim_{x\rightarrow r_{+}} \{ y^{\prime}(x;\tau_{1})\cdot y(x;\tau_{2}) - y^{\prime}(x;\tau_{2})\cdot y(x;\tau_{1}) \} =0
$$
and
$$
\lim_{x\rightarrow s_{-}} \{ y^{\prime}(x;\tau_{1})\cdot y(x;\tau_{2}) - y^{\prime}(x;\tau_{2})\cdot y(x;\tau_{1}) \} =0.
$$
If $\eta\in(r,s)$ exists such that $f(\eta;\tau_1)=f(\eta;\tau_2)$ and
\begin{itemize}
\item $f(x;\tau_{2})-f(x;\tau_{1})<0$ for $x\in(r,\eta)$, $f(x;\tau_{2})-f(x;\tau_{1})>0$ for $x\in(\eta,s)$, then $x_{k}(\tau_{1})<x_{k}(\tau_{2})$ for every $k=1,\ldots,n$;
\item $f(x;\tau_{2})-f(x;\tau_{1})>0$ for $x\in(r,\eta)$, $f(x;\tau_{2})-f(x;\tau_{1})<0$ for $x\in(\eta,s)$, then $x_{k}(\tau_{1})>x_{k}(\tau_{2})$ for every $k=1,\ldots,n$.
\end{itemize}
\end{THEO}

To be able to use the above theorem for the analysis of the zeros of $X_{1}$-Laguerre polynomials, we consider the  Sturm-Liouville form

\begin{equation}
\label{ODEu}
u_{n}^{\prime\prime}(x)+\lambda_{n,4}(x)u_{n}(x)=0,\ \ \ u_{n}(x)=\frac{e^{-x/2} x^{(\alpha +1)/2}}{x+\alpha}\widehat{L}_{n}^{(\alpha)}(x),
\end{equation}
of the differential equation (\ref{eqLaguerre}), where
$$
\lambda_{n,4}(x)=\frac{ -x^4 +A_{3}(n,\alpha)x^3 +A_{2}(n,\alpha)x^2+A_{1}(n,\alpha)x+A_{0}(n,\alpha)}{ 4x^2 (x+\alpha)^2},
$$
with
\begin{eqnarray*}
A_{3}(n,\alpha)&=&2(2n-1),\\
A_{2}(n,\alpha)&=&2\alpha^{2}+4(2n-1)\alpha-3\\
A_{1}(n,\alpha)&=&2\alpha[(2n-1)\alpha+3],\\
A_{0}(n,\alpha)&=&\alpha ^2 (1-\alpha^{2}).
\end{eqnarray*}
The details of the transformation between the Sturm-Liouville form and (\ref{eqLaguerre}) can be found in Section 1.8 of \textbf{\cite{szego}}.  Finally we prove the following simple technical lemma:
\begin{lemma}\label{JacobiInequality}
Let $\{P_{n}^{(\alpha,\beta)}(x)\}_{n=0}^{\infty}$ be the Jacobi OPS. Then $P_{n}^{(\alpha,\beta)}(x)<P_{n+1}^{(\alpha,\beta)}(x)$ for every $x>1$ and $n$ is a positive integer.
\end{lemma}
\begin{proof}
Given $n\geq1$, denote by $x_{n,k}$, $k=1,\ldots,n$, the zeros of $P_{n}^{(\alpha,\beta)}(x)=a_{n,n}x^{n}+\ldots+a_{n,0}$ in increasing order. By the interlacing property and $a_{n,n}>0$, the polynomial $Q_{n+1}(x):=P_{n+1}^{(\alpha,\beta)}(x)-P_{n}^{(\alpha,\beta)}(x)$ has $n$ zeros in $(x_{n+1,1},x_{n+1,n+1})$ and
$$
P_{n+1}^{(\alpha,\beta)}(x_{n+1,n+1}+\epsilon)<P_{n}^{(\alpha,\beta)}(x_{n+1,n+1}+\epsilon),
$$
for $\epsilon>0$ sufficiently small. On the other hand, we have
$$
P_{n}^{(\alpha,\beta)}(1)=\binom{n+\alpha}{n}<\binom{n+1+\alpha}{n+1}=P_{n+1}^{(\alpha,\beta)}(1)
$$
which implies that $Q_{n+1}(x)$ has $n+1$ zeros in $(-1,1)$. Since the former inequality yields $Q_{n+1}(1)>0$, then $Q_{n+1}(x)>0$ for every $x>1$.
\end{proof}


\section{Proofs of main results}
\begin{proof}[\textbf{Proof of Theorem \ref{JacobiExceptionalZeroTheo}}]
First we show that there is a partition in $(b,c]$ which divides the set of \textit{exceptional} zeros of the sequence $\{\widehat{P} _{n}^{(\alpha,\beta)}(x)\}_{n=1}^{\infty}$ in $(b,c]$ for certain values of $\alpha$ and $\beta$. The \textit{exceptional} zero of the $n$th polynomial of the sequence belongs to $(b,\gamma_{n}b)$, where $\gamma_{n}:=(2n+\alpha+\beta)/(2n-2+\alpha+\beta)$, for $n\geq2$. In fact, \eqref{relacaoJacobi} implies
$$
\widehat{P} _{n}^{(\alpha , \beta )}(x)=f_{1}(x)P_{n-1}^{(\alpha , \beta )}(x)-\frac{P _{n-2}^{(\alpha , \beta )}(x)}{2n-2+\alpha+\beta},
$$
where
$$
f_{1}(x):=-\frac{1}{2}(x-b)+\frac{b}{2n-2+\alpha+\beta},
$$
and $f_{1}(x)=0$ if and only if $x=\gamma_{n}b$. Since the leading coefficient of the Jacobi polynomial is positive and the \textit{exceptional} zeros of $X_{1}$-Jacobi polynomials belong in $(b,c]$, we have that $f_{1}(x)P_{n-1}^{(\alpha , \beta )}(x)$ and $-P _{n-2}^{(\alpha , \beta )}(x)/(2n-2+\alpha+\beta)$ have opposite signs in $(b,\gamma_{n}b)$ and the same sign in $(\gamma_{n}b,c]$. Therefore, the \textit{exceptional} zero of $\widehat{P} _{n}^{(\alpha,\beta)}(x)$ belongs to $(b,\gamma_{n}b)$. Note that $\gamma_{1}b=c=\widehat{x} _{1,1}^{(\alpha , \beta )}$ and $\gamma_{n}$ decreases strictly in $n$. Hence, the \textit{exceptional} zeros $\widehat{x} _{2,2}^{(\alpha , \beta )}$ and $\widehat{x} _{1,1}^{(\alpha , \beta )}$ are distinct for any $\alpha$ and $\beta$. On the other hand, $\widehat{P} _{n}^{(\alpha , \beta )}(\gamma_{n+1}b)> 0$ for $\alpha$ and $\beta$ sufficiently close. In fact, by formula \eqref{relacaoJacobi} and Lemma \ref{JacobiInequality}, we have
$$
\widehat{P} _{n}^{(\alpha , \beta )}(\gamma_{n+1}b)=\frac{b}{2n-2+\alpha+\beta}\Big[2\frac{P_{n-1}^{(\alpha , \beta )}(\gamma_{n+1}b)}{2n+\alpha+\beta}-\frac{\beta-\alpha}{\beta+\alpha}P_{n-2}^{(\alpha , \beta )}(\gamma_{n+1}b)\Big]> 0.
$$
Then, subject to this condition that $\alpha$ and $\beta$ are sufficiently close, we have $\gamma_{n+1}b< \widehat{x} _{n,n}^{(\alpha , \beta )} <\gamma_{n}b$. Therefore
$$
b<\ldots<\widehat{x} _{n,n}^{(\alpha,\beta)}<\ldots<\widehat{x} _{2,2}^{(\alpha , \beta )}<\widehat{x} _{1,1}^{(\alpha , \beta )}=c.
$$

Since the zeros of $\widehat{P} _{n}^{(\alpha , \beta )}(x)$ are continuous functions of the parameters $\alpha$ and $\beta$, we find that the inequality \eqref{zerosinequalities} is initially satisfied when the parameters are sufficiently close. If we change the parameters from the previous condition on $\alpha$ and $\beta$, the inequality \eqref{zerosinequalities} still remains valid. Suppose that $(\alpha,\beta)$ and $k$ exist such that
$$
\ldots<\widehat{x} _{k+2,k+2}^{(\alpha , \beta )}<\widehat{x} _{k+1,k+1}^{(\alpha , \beta )}=\widehat{x} _{k,k}^{(\alpha , \beta )}<\widehat{x} _{k-1,k-1}^{(\alpha , \beta )}<\ldots.
$$
Let us set $\xi:=\widehat{x} _{k,k}^{(\alpha , \beta )}$. Since $\widehat{x} _{2,2}^{(\alpha , \beta )}\neq\widehat{x} _{1,1}^{(\alpha , \beta )}$ for any values of $\alpha$ and $\beta$, then \eqref{relacaoJacobi01} yields
$$
0>-\frac{1}{4}(\xi-b)^{2}P_{k-1}^{(\alpha , \beta )}(\xi)=h_{k-1}\widehat{P} _{k-1}^{(\alpha , \beta )}(\xi)>0,\ \ k\geq2,
$$
which is a contradiction.

We know that $\widehat{L} _{1}^{(\alpha  )}(x)=-(x+\alpha+1)$ and the exceptional zeros of $\widehat{L} _{n}^{(\alpha  )}(x)$ are smaller than $-\alpha$ (see \textbf{\cite{GoKaMi09}}). To establish the lower bound for these zeros it suffices to analyze the formula \eqref{relacaoLaguerre}. Note that $-\widehat{L} _{1}^{(\alpha  )}(x)L_{n-1}^{(\alpha  )}(x)$ and $L_{n-2}^{(\alpha  )}(x)$ have opposite signs in $(-\alpha-1,0)$ and the same sign in $(-\infty,-\alpha-1)$ for $n\geq 2$. Then the exceptional zeros belong to $[-\alpha-1,-\alpha)$.

Inequality \eqref{LaguerreZerosInequalities} is an immediate consequence of \eqref{zerosinequalities} and Theorem \ref{JacobiLaguerreConvergence}. Indeed, since
$$
\widehat{x} _{n+1,n+1}^{(\alpha,\beta)}<\widehat{x} _{n,n}^{(\alpha,\beta)}\ \  \mathrm{if\ and\ only\ if}\ \
\frac{\beta(1-\widehat{x} _{n,n}^{(\alpha,\beta)})}{2}<\frac{\beta(1-\widehat{x} _{n+1,n+1}^{(\alpha,\beta)})}{2}
$$
for every $n\geq1$. Taking $\beta$ sufficiently large, we obtain the inequality \eqref{LaguerreZerosInequalities}.
\end{proof}

\begin{proof}[\textbf{Proof of Theorem \ref{InterlacingMonotZerosLaguerre}}]
The first statement is a consequence of Theorem \ref{SturmTheoClassic}. Indeed, note that
$$
\lambda_{n+1,4}(x)-\lambda_{n,4}(x)=\frac{1}{x}
$$
and the solution $u_{n+1}(x)$ satisfy the conditions \eqref{a} and \eqref{b} for $r=0$ and $s=\infty$.

To establish the monotonicity of the \textit{regular} zeros, note that
$$
\frac{\partial\lambda_{n,4}(x)}{\partial\alpha}=\frac{p_{n,4}(x;\alpha)}{2 x^2 (\alpha +x)^3}=\frac{x^{4}+B_{3}(\alpha)x^{3}+B_{2}(\alpha)x^{2}+B_{1}(\alpha)x+B_{0}(\alpha)}{2 x^2 (x+\alpha)^3},
$$
with
\begin{eqnarray*}
B_{3}(\alpha)&=& 2\alpha,\\
B_{2}(\alpha)&=& 6,\\
B_{1}(\alpha)&=&-2\alpha(1+\alpha^{2}),\\
B_{0}(\alpha)&=&-\alpha^{4}.
\end{eqnarray*}
Obviously, the sequence $\{B_{k}(\alpha)\}_{k=0}^{3}$ has just one sign change for every $\alpha>0$. Therefore, by the classical Descartes' rule of signs, $p_{n,4}(x;\alpha)$ has just one positive zero and the conclusion follows from the first statement of Theorem \ref{SturmTheoNew}.
\end{proof}

\begin{proof}[\textbf{Proof of Theorem \ref{JacobiLaguerreConvergence}}]
The proof is based on the Gaussian hypergeometric representation of the Jacobi polynomials
\begin{align}\label{jacobi}
\begin{split}
P^{(\alpha,\beta)}_{n}\Big(1-\frac{2x}{\beta}\Big)&=\frac{(\alpha+1)_{n}}{n!}\, _2F_1\left(-n,\alpha +\beta +n+1;\alpha +1;\frac{x}{\beta }\right)\\
&=\sum_{k=0}^{n}\frac{R_{n,k}^{(\alpha)}(x)}{\beta^{k}},
\end{split}
\end{align}
where $R_{n,0}^{(\alpha)}(x)=L^{(\alpha)}_{n}(x)$ (see \textbf{\cite[Section 5.3]{szego}} and \textbf{\cite{Eliel}} for the remaining $R_{n,k}^{(\alpha)}(x)$).

For $\alpha<\beta$, we have
$$
b=\frac{\beta+\alpha}{\beta-\alpha}=\Big(1+\frac{\alpha}{\beta}\Big)\sum_{k=0}^{\infty}\Big(\frac{\alpha}{\beta}\Big)^{k}=
1+2\sum_{k=1}^{\infty}\Big(\frac{\alpha}{\beta}\Big)^{k}.
$$
Applying \eqref{jacobi} in \eqref{relacaoJacobi}, we obtain
\begin{eqnarray*}
\frac{\widehat{P}^{(\alpha,\beta)}_{n}(1-2x/\beta)}{(2n-2+\alpha+\beta)^{-1}}
&=&\Big\{(2n-2+\alpha+\beta)\Big[\frac{x+\alpha}{\beta}+\sum_{k=2}^{\infty}\Big(\frac{\alpha}{\beta}\Big)^{k}\Big]
+1+2\sum_{k=1}^{\infty}\Big(\frac{\alpha}{\beta}\Big)^{k}\Big\}\\
& &\sum_{k=0}^{n-1}\frac{R_{n-1,k}^{(\alpha)}(x)}{\beta^{k}}-\sum_{k=0}^{n-2}\frac{R_{n-2,k}^{(\alpha)}(x)}{\beta^{k}}\\
&=&\Big[(2n-2+\alpha+\beta)\frac{x+\alpha}{\beta}+1\Big]R_{n-1,0}^{(\alpha)}(x)-R_{n-2,0}^{(\alpha)}(x)+O(\beta^{-1})\\
&=&(x+\alpha+1)L_{n-1}^{(\alpha)}(x)-L_{n-2}^{(\alpha)}(x)+O(\beta^{-1})\\
&=&-\widehat{L}_{n}^{(\alpha)}(x)+O(\beta^{-1}).
\end{eqnarray*}

Setting $y=\sqrt{2\alpha}x+\alpha$ and manipulating of formula \eqref{relacaoLaguerre}, we obtain
\begin{eqnarray*}
  (n-1)!\Big(\frac{2}{\alpha}\Big)^{(n-1)/2}\frac{\widehat{L}_{n}^{(\alpha)}(y)}{y-\widehat{x}^{(\alpha)}_{n,1}}&=&
  -\frac{y+\alpha+1}{y-\widehat{x}^{(\alpha)}_{n,1}}(n-1)!\Big(\frac{2}{\alpha}\Big)^{(n-1)/2}L^{(\alpha)}_{n-1}(y)+\\
   & &\frac{(n-1)\sqrt{2}}{(y-\widehat{x}^{(\alpha)}_{n,1})\sqrt{\alpha}}(n-2)!\Big(\frac{2}{\alpha}\Big)^{(n-2)/2}L_{n-2}^{(\alpha)}(y).
\end{eqnarray*}
Theorem \ref{JacobiExceptionalZeroTheo} assures that $-\widehat{x}^{(\alpha)}_{n,1}\in(\alpha,\alpha+1]$ for every $n\geq 1$. Therefore, considering the limit in the above equation and \cite{Calogero}, we obtain
$$
(n-1)!\lim_{\alpha\to\infty}\Big(\frac{2}{\alpha}\Big)^{(n-1)/2}\frac{\widehat{L}_{n}^{(\alpha)}(y)}{y-\widehat{x}^{(\alpha)}_{n,1}}=- H_{n-1}(x).
$$
\end{proof}

\section{Final Remarks}
$X_{1}$-Laguerre polynomials treated in this manuscript is the particular case of type I $X_{m}$-Laguerre polynomials denoted by $L_{m,n}^{I,\alpha}(x)$. There is a practical way in \textbf{\cite{BonneuxArno}} to construct these polynomials by
\begin{equation}\label{LaguerreTypeOne}
L_{m,n}^{I,\alpha}(x)=
\begin{vmatrix} L_{m}^{(\alpha)}(-x) & -L_{n-m-1}^{(\alpha)}(x) \\ L_{m}^{(\alpha-1)}(-x) & L_{n-m}^{(\alpha-1)}(x) \end{vmatrix},
\end{equation}
for $n\geq m$. If we denote $z_{m,n,k}$, for $k=1,\ldots,m$, the \textit{exceptional} zeros of $L_{m,n}^{I,\alpha}(x)$ in decreasing order, Proposition 3.2 of \textbf{\cite{Paco}} assures that the $z_{m,n,k}$ are all simple and located in $(-\infty,0)$ and, further, by  Proposition 3.4 these zeros  converge to the zeros of $L_{m}^{(\alpha-1)}(-x)$ as $n$ tends to infinity. Theorem \ref{JacobiExceptionalZeroTheo} is the particular case when we have $m=1$. It is natural to investigate if, for the general case when $m>1$, the \textit{exceptional} zeros of $L_{m,n}^{I,\alpha}$ satisfy  inequality relations  like  \eqref{LaguerreZerosInequalities}. We do not have any proof of this, but the following numerical results suggest that this might be true.
\begin{center}
\begin{tabular}{ |p{1cm}||p{2.5cm}|p{2.5cm}|p{2.5cm}||p{2.5cm}| }
 \hline
 \multicolumn{5}{|c|}{Table of \textit{Exceptional} zeros for $m$=4 and $\alpha=1$ } \\
 \hline
 $z_{m,n,k}$ & $n=6$ & $n=10$ & $n=14$ & $-x_{m,k}^{(\alpha-1)}$\\
 \hline
 $z_{m,n,1}$ & -0.62239 & -0.53595 & -0.49680 & -0.32254 \\
 $z_{m,n,2}$ & -2.36155 & -2.20066 & -2.12323 & -1.74576 \\
 $z_{m,n,3}$ & -5.47132 & -5.24298 & -5.12778 & -4.53662 \\
 $z_{m,n,4}$ & -10.6643 & -10.3770 & -10.2244 & -9.39507 \\
 \hline
\end{tabular}
\end{center}
\begin{problem}
Let $L_{m,n}^{I,\alpha}(x)$ the type I $X_{m}$-Laguerre polynomials defined in \eqref{LaguerreTypeOne} and denote $z_{m,n,k}=z_{m,n,k}^{(\alpha)}$, $k=1,\ldots,m$, the \textit{exceptional} zeros of them in decreasing order. Then we have
$$
\ldots<z_{m,n-1,k}<z_{m,n,k}<z_{m,n+1,k}<\ldots <-x_{m,k}^{(\alpha-1)}
$$
for each $k=1,\ldots, m$.
\end{problem}

Another interesting question involving the identities \eqref{JacobiLaguerreFormula} and \eqref{LimX1LaguerreZeroToHermite} is about a monotonicity with respect to the corresponding parameter. It is known from \textbf{\cite{DiFer07,DiFer09}} that the expressions
\begin{equation}\label{remark}
\frac{f_{n}(\alpha,\beta)\big(1-x_{n,k}^{(\alpha,\beta)}\big)}{2} \ \ \mbox{and} \ \ \frac{x_{n,k}^{(\alpha)}-(2n+\alpha-1)}{\sqrt{2(n+\alpha-1)}}
\end{equation}
are strictly increasing functions and the last one converges to $h_{n,k}$ for each $k=1,\ldots,n$ when we take $\alpha$ as sufficiently large, where
\begin{equation}\label{f}
f_{n}(\alpha,\beta)=2n^{2}+2n(\alpha+\beta+1)+(\alpha+1)(\beta+1)
\end{equation}
and $x_{n,k}^{(\alpha)}$, $x_{n,k}^{(\alpha,\beta)}$, $h_{n,k}$ denote the zeros of the classical Laguerre, Jacobi and Hermite orthogonal polynomials, respectively. The numerical tests confirm the similar monotonic behavior of the expression as the second one in \eqref{remark} involving the \textit{regular} zeros of $X_{1}$-Laguerre polynomials and  the first expression of \eqref{remark} involving all zeros of $X_{1}$-Jacobi polynomials.

The above observations still remain open to being proved.

\section*{Acknowledgements}
The author is grateful to Prof. Alagacone S. Ranga for reading a preliminary version of this note and to Prof. Dimitar K. Dimitrov for his constructive suggestions.

\end{document}